\documentclass[reqno]{amsart}

\usepackage{mathrsfs}
\usepackage{amsthm}
\usepackage{amscd}
\usepackage{amsmath}
\usepackage{amssymb}
\usepackage{extarrows}
\usepackage[colorlinks]{hyperref}
\usepackage{url}

\theoremstyle{definition}
\newtheorem{defn}{Definition}[section]

\theoremstyle{plain}
\newtheorem{thm}{Theorem}[section]
\newtheorem*{thm*}{Theorem}

\newtheorem{prop}[thm]{Proposition}
\newtheorem{lem}[thm]{Lemma}

\newtheorem{clm}{Claim}

\theoremstyle{remark}
\newtheorem*{rmk}{Remark}

\newcommand{\ve}{\varepsilon}

\newcommand{\mbbn}{\mathbb{N}}

\newcommand{\mbbr}{\mathbb{R}}

\newcommand{\mbbz}{\mathbb{Z}}

\newcommand{\mcld}{\mathcal{D}}

\newcommand{\mclg}{\mathcal{G}}

\newcommand{\mcli}{\mathcal{I}}

\newcommand{\mclm}{\mathcal{M}}

\newcommand{\mclo}{\mathcal{O}}
\newcommand{\mclp}{\mathcal{P}}

\newcommand{\mcls}{\mathcal{S}}

\newcommand{\mclu}{\mathcal{U}}
\newcommand{\mclv}{\mathcal{V}}
\newcommand{\mclw}{\mathcal{W}}

\newcommand{\whK}{\widehat{K}}

\newcommand{\dif}{\:\!\mathrm{d}}

\newcommand{\supp}{\mathrm{supp}}

\DeclareMathOperator{\dist}{dist} 

\DeclareMathOperator{\lip}{lip}

\title{On ergodic optimization for unimodal maps}

\author[B. Gao]{Bing Gao}
\address{Bing Gao: School of Mathematics, Hunan University, Changsha 410082, China}
\email{binggaomath@outlook.com}

\author[R. Gao]{Rui Gao}

\address{Rui Gao: School of Mathematical Sciences, Qufu Normal University, Jining 273165, China}
\email{gaoruimath@qfnu.edu.cn}

\begin{document}

\begin{abstract}
In this article, we show that for a typical non-uniformly expanding unimodal map, the unique maximizing measure of a generic Lipschitz function is supported on a periodic orbit.
\end{abstract}

\maketitle

\section{Introduction}\label{se:intro}


\subsection{Backgrounds} 
Let $X$ be a compact metric space and let $C(X)$ denote the collection of real-valued continuous functions on $X$. Let $T:X\to X$ be a continuous map and let $\mclm_T$ denote the collection of $T$-invariant Borel probability measures on $X$. Given $\phi\in C(X)$,   denote
\begin{equation}\label{eq:max int}
  \beta(\phi):=\max_{\mu\in \mclm_T}\int \phi\dif\mu.
\end{equation}
If $\mu\in  \mclm_T$ attains the maximum value in \eqref{eq:max int}, then we say that $\mu$ is a {\bf maximizing measure} of $\phi$, or simply $\mu$ {\bf maximizes} $\phi$. 
If a maximizing measure $\mu$ is supported on a periodic orbit $\mclo$ of $(X,T)$, then we also say that $\phi$ is {\bf maximized} by $\mclo$. From a dynamical view, the {\bf ergodic optimization} problem of $(X,T)$ concerns typical behavior of maximizing measures of $\phi$ for $\phi$ chosen in a suitable subspace of $C(X)$. 

Since the pioneering works of Hunt and Ott \cite{HO96L,HO96E}, and of Yuan and Hunt \cite{YH96}, it is conjectured that, when the system $(X,T)$ has strongly enough hyperbolic behavior (e.g., uniformly expanding, Anosov), for typical $\phi$ chosen in a real Banach space consisting of sufficiently regular  functions (e.g., Lipschitz, $C^1$),  $\phi$ will be (uniquely) maximized by a periodic orbit. Here ``typical" could be understood in either a probabilistic (namely, ``prevalent") or a topological (namely, ``generic")  way. Following Jenkinson \cite{Jen19}, we call this conjecture the {\bf typically periodic optimization}  ({\bf TPO} for short) conjecture. In an early stage of researches on this kind of problems, a remarkable result of Bousch \cite{Bou00} showed that for the circle doubling map $x\mapsto 2x, x\in \mbbr/\mbbz$, TPO (in both  probabilistic and topological senses) holds for the two dimensional space spanned by $\cos(2\pi x)$ and  $\sin(2\pi x)$. Another remarkable work around the same period is \cite{CLT01}, in which Contreras, Lopes and Thieullen proved a version of the topological TPO conjecture for expanding maps of the circle, and developed a series of tools for ergodic optimization research originated from Lagrangian dynamics. 
 
 As an important progress in ergodic optimization, Contreras \cite{Con16} proved that when $(X,T)$ is uniformly expanding, topological TPO holds for Lipschitz functions. More recently, Huang et al. \cite{HLMXZ19} generalized this result to uniformly hyperbolic systems and to $C^1$ functions. Our motivation in this article is attempting to generalize the result of  Contreras mentioned above to some non-uniformly expanding systems, so we shall only focus on the topological TPO conjecture, and omit ``topological" in this term for simplicity. For discussions on the probabilistic TPO conjecture we refer to Bochi \cite{Boc18} and references therein, and for other aspects in ergodic optimization we refer to \cite{Boc18,Jen19}.

\subsection{Statement of results}

The results proved by Contreras \cite{Con16} and by Huang et al. \cite{HLMXZ19} are slightly stronger than the statement in the TPO conjecture mentioned before. To quote them (see the original papers \cite{Con16,HLMXZ19} or Theorem~\ref{thm:Contreras} below for the precise statements) and to state our results, let us introduce the following terminology for convenience. Let $X$ be a compact metric space and let $Lip(X)$ be the Banach space of real-valued Lipschitz functions on $X$ endowed with the Lipschitz norm. For a continuous map $T:X\to X$, we say that {\bf Lipschitz TPO holds} for $(X,T)$, if there exists a dense (and open) subset $\mclg$ of $Lip(X)$, such that for any $\phi\in \mclg$, there exists an open neighborhood $\mclv$ of $\phi$ in $Lip(X)$ and a periodic orbit $\mclo$ such that each function in $\mclv$ is (uniquely) maximized by $\mclo$. 

When the dynamical behavior of $(X,T)$ is beyond uniformly expanding/hyperbolic (but still non-uniformly expanding/hyperbolic), there is little known for the TPO conjecture in the literatures. In this article, as one of the first attempts to attacking this kind of problems, the results in \cite{Con16,HLMXZ19} will be our cornerstone. Therefore, we shall restrict ourselves to working only on Lipschitz functions and choose to consider systems not far from being uniformly expanding/hyperbolic. For simplicity, we shall focus on one of the simplest situations: the phase space $X$ will be a compact interval and the map $T:X\to X$ will be a unimodal map with strong expanding properties.  

To state our results, we shall impose some conditions on unimodal maps that are extensively studied in one-dimensional dynamics. Since their descriptions need quite a few words, we choose to postpone introducing them in \S~\ref{sse:unimodal} and to state the results first. Our main result in this article is the following.

\begin{thm}\label{thm:unimodal}
Let $T:X\to X$ be either a piecewise expanding unimodal map or an S-unimodal map satisfying the Collet-Eckmann condition.
Let $\omega(c)$ denote the $\omega$-limit set of the turning point $c$, and suppose that at least one of the following assumptions is satisfied:
\begin{itemize}
  \item the subsystem $(\omega(c),T)$ is not uniquely ergodic;
  \item or $\omega(c)$ contains a periodic orbit;
  \item or $c\notin\omega(c)$.
\end{itemize}
Then Lipschitz TPO holds for $(X,T)$.
\end{thm}

\begin{rmk} Let us make some comments on the statement of Theorem~\ref{thm:unimodal}.
\begin{itemize}
  \item Similar to the situations in \cite{Con16,HLMXZ19} among others, our proof of Theorem~\ref{thm:unimodal} works equally well if ``Lipschitz" is generalized to ``H\"{o}lder" in the statement. It should also be not hard to generalize this result to $C^1$ functions or to multimodal maps with suitable conditions imposed on the maps. However, for simplicity we choose to state our result in the current form.   
  \item We have to admit that our assumption on  $\omega(c)$ in Theorem~\ref{thm:unimodal} is mainly for technical reasons (see the proof of Proposition~\ref{prop:reducible} for details), but at this moment we have no idea how to get fully rid of it.   
  \item As a supplement of Theorem~\ref{thm:unimodal}, in \S~\ref{se:circle} we shall discuss the TPO problem in Lipschitz topology for a different kind of one-dimensional non-uniformly expanding maps: circle covering maps with an indifferent fixed point.    
\end{itemize}
\end{rmk}

%
With the help of classic results in one-dimensional dynamics, we can easily obtain the following corollary from Theorem~\ref{thm:unimodal}.

\begin{thm}\label{thm:exclu}
The following hold.
\begin{itemize}
  \item [(1)] For $a\in (0,4]$, let $Q_a(x)=ax(1-x), \forall x\in [0,1]$. Let $a_*$ be the supremum of $a\in(0,4]$ such that $Q_a:[0,1]\to [0,1]$ has topological entropy zero.
  Then for Lebesgue a.e. $a\in [a_*,4]$, Lipschitz TPO holds for $Q_a:[0,1]\to [0,1]$. 
  \item [(2)] For $a\in [1,2]$, let $T_a(x)=a(1-|x-1|),\forall x\in [0,2]$. Then for Lebesgue a.e. $a\in [1,2]$, Lipschitz TPO holds for $T_a:[0,2]\to [0,2]$. 
\end{itemize}
\end{thm}

\begin{rmk}
  In assertion~(1) of Theorem~\ref{thm:exclu} the parameter $a_*\approx 3.57$ is known as the Feigenbaum–Coullet–Tresser parameter. It is well-known that when $a\in(0,a_*)$, $Q_a:[0,1]\to [0,1]$ has only finitely many periodic orbits and $\omega(x)$ coincides with one of them for every $x\in [0,1]$; see, for example, \cite[page~8]{SKSF97}, where $a_*$ is denoted as $\lambda^*$. That is why we restrict the range of $a$ to $[a_*,4]$.
\end{rmk}

\subsection{Comments on the proof}

Let us try to explain our main idea in proving Theorem~\ref{thm:unimodal} and to highlight the difference between the situation in our setting and in the uniformly expanding case. 

In previous studies on the TPO conjecture under uniformly expanding/hyperbolic settings, including \cite{Bou00,CLT01,Con16,HLMXZ19} cited before, the concept of sub-action plays a fundamental role. Let us recall its definition first. Let $T:X\to X$ be a continuous self-map on a compact metric space $X$. Following \cite{CLT01}, given $\phi\in C(X)$, $\psi\in C(X)$ is called a {\bf sub-action} of $\phi$ if the following holds on $X$:
\begin{equation}\label{eq:sub defn}
  \phi \le \psi\circ T -\psi + \beta(\phi).
\end{equation}
As is well-known, when $(X,T)$ is uniformly expanding, for $\alpha\in (0,1]$, any $\alpha$-H\"{o}lder function on $X$ admits some $\alpha$-H\"{o}lder sub-action. For $\phi\in C(X)$, among other things, an advantage of existence of sub-actions is its implication of the following {\bf subordination principle}: if $\mu\in \mclm_T$ maximizes $\phi$, then for any $\nu\in \mclm_T$ with $\supp\nu\subset \supp\mu$, $\nu$ is also a maximizing measure of $\phi$. Here $\supp\mu$ denotes the support of $\mu$ as usual; more precisely,
\begin{equation}\label{eq:supp}
  \supp \mu:=\{x\in X: \mu(B)>0~\mbox{holds for any open neighborhood $B$ of}~x\}. 
\end{equation}
For further discussions on sub-actions in ergodic optimization we refer to \cite{CLT01,Jen19}.

Unfortunately, for non-uniformly expanding/hyerbolic dynamical systems, sub-actions of a H\"{o}lder (or even Lipschitz) function may not exist, as indicated by Morris~\cite[Theorem~2]{Mor09} and by Garibaldi and Inoquio-Renteria~\cite[Theorem~1]{GI20}.  For this reason among others, the approaches in \cite{Con16,HLMXZ19} cannot work under our setting. On the other hand, another result of Morris \cite[Theorem~1]{Mor07} (see also Lemma~\ref{lem:subord} below) shows that a weaker assumption than existence of sub-actions is sufficient to imply the subordination principle mentioned above. 
With the help of this result, our main idea to prove Theorem~\ref{thm:unimodal} is quite simple. We first show that the assumption in Lemma~\ref{lem:subord} is satisfied (in a locally uniform way) under our setting; this is given by  Proposition~\ref{prop:subord}. Based on this fact, then we show that our problem can be reduced to the uniformly expanding case, so the results in \cite{Con16,HLMXZ19} become applicable; see Proposition~\ref{prop:reducible} and Lemma~\ref{lem:reducible} for details.

\subsection{Conventions and notations}

In the rest of this article, we shall adopt the following conventions and notations. 

\begin{itemize}
  \item For a metric space $X$, the distance function of $X$ will be denoted by $\dist(\cdot,\cdot)$; for a non-empty subset $E$ of $X$, $\dist(\cdot,E)$ denotes the distance of a point to $E$, i.e., $\dist(x,E)=\inf_{y\in E}\dist(x,y),\forall x\in X$.
  \item For a metric space $X$, the only function space on $X$ we are concerned with is $Lip(X)$, the collection of real-valued Lipschitz functions on $X$. For $\phi\in Lip(X)$, let $\lip(\phi)$ denote its best Lipschitz constant, i.e., 
      \[
      \lip(\phi)=\sup_{x\ne y}\frac{|\phi(x)-\phi(y)|}{\dist(x,y)}.
      \] 
  When $X$ is compact, $Lip(X)$ is considered as a Banach space endowed with the Lipschitz norm; we shall not use the Lipschitz norm on $Lip(X)$ explicitly. 
  \item For a self-map $T:X\to X$ on a space $X$ and for $\phi:X\to \mbbr$, let $\mcls_n\phi:=\sum_{0\le k<n}\phi\circ T^k$ denote the $n$-th Birkhoff sum of $\phi$ for each $n\ge 1$.
  \item For a continuous self-map $T:X\to X$ on a compact metric space $X$, given $x\in X$, let $\omega(x)$ denote the $\omega$-limit set of $x$. 
  \item By saying a closed interval in $\mbbr$, we mean a non-degenerate compact interval. The notations $[a,b]$ and $[b,a]$ are used to denote the same closed interval with distinct endpoints $a,b\in\mbbr$. For an open interval $(a,b)$ in $\mbbr$ we adopt similar conventions. 
  \item For a subset $E$ of $\mbbr$ (respectively $\mbbr/\mbbz$), let $\overline{E}$ denote its closure in $\mbbr$ (respectively $\mbbr/\mbbz$). 
  \item For a Borel subset $E$ of either $\mbbr$ or $\mbbr/\mbbz$, let $|E|$ denote its Lebesgue measure. 
  \item For a finite set $S$, let $\# S$ denote its cardinality.
\end{itemize} 

The rest of this article is organized as follows. \S~\ref{se:pre} serves as a preliminary part. In \S~\ref{se:unimodal} we prove Theorem~\ref{thm:unimodal} and Theorem~\ref{thm:exclu}. In \S~\ref{se:circle} we prove a parallel result to Theorem~\ref{thm:unimodal} for certain non-uniformly expanding circle maps. In \S~\ref{app:lock} we provide a proof of Lemma~\ref{lem:locking}. In \S~\ref{app:Morris} we provide a proof of Lemma~\ref{lem:Morris}.

\section{Preliminaries}\label{se:pre}

In this section we summarize known results that will be used to prove Theorem~\ref{thm:unimodal} in the next section. In \S~\ref{sse:Contreras}, we introduce Lemma~\ref{lem:reducible} as a corollary of results in \cite{Con16,HLMXZ19}, which is the starting point of our argument. In \S~\ref{sse:Morris}, we introduce Lemma~\ref{lem:subord} borrowed from \cite{Mor07}, which is another building block in our argument. In \S~\ref{sse:unimodal} we collect necessary definitions and facts on unimodal maps that will be used in the proof of Theorem~\ref{thm:unimodal}. 
 
\subsection{On Lipschitz typically periodic optimization}\label{sse:Contreras}

In this subsection we review the results in \cite{Con16,HLMXZ19} and deduce Lemma~\ref{lem:reducible} from them as a corollary.

Following \cite[Chapter~4]{PU10}, for a Lipschitz self-map $T:X\to X$ on a compact metric space $X$, $T$ is called {\bf distance-expanding}, or simply {\bf expanding}, if there exists an integer $n\ge 1$ and $\eta>0,\lambda>1$ such that for $x,y\in X$ with $\dist(x,y)\le 2\eta$ we have $\dist(T^n x, T^n y)\ge \lambda\cdot\dist(x,y)$.  We refer to  
\cite[Chapter~4]{PU10} for useful properties of a map $T:X\to X$ which is both open and expanding. When a map $T:X\to X$ is surjective, open and expanding,  the required assumptions on $(X,T)$ in either \cite{Con16} or \cite{HLMXZ19} are satisfied. Therefore, according to Contreras \cite{Con16} and Huang et al. \cite{HLMXZ19}, the following Theorem~\ref{thm:Contreras} holds; in particular, the ``moreover" part is an immediate corollary of Proposition~3.1 and Proposition~3.2 in \cite{HLMXZ19}. See also \cite{Boc19} for a presentation of the proof in \cite{HLMXZ19} under the settings of Theorem~\ref{thm:Contreras} below.

\begin{thm}[\cite{Con16,HLMXZ19}] \label{thm:Contreras}
Let $X$ be a compact metric space. Let $T:X\to X$ be a Lipschitz map which is surjective, open and expanding. Then Lipschitz TPO holds for $(X,T)$. Moreover, for any $\phi\in Lip(X)$ and any $\ve>0$, there exists a periodic orbit $\mclo$ such that $\phi-\ve\cdot\dist(\cdot,\mclo)$ is uniquely maximized by $\mclo$.
\end{thm}

To make Theorem~\ref{thm:Contreras} easier to use for our purpose, we also need the following fact, which was already known to Yuan-Hunt in \cite[Remark~4.5]{YH96} (under hyperbolicity hypothesis on $(X,T)$), and a proof was presented by Bochi and Zhang \cite{BZ15}. As suggested by the referee, we shall give a proof of this result in \S~\ref{app:lock} for completeness. 

\begin{lem}[\cite{YH96,BZ15}]\label{lem:locking}
  Let $X$ be a compact metric space and let $T:X\to X$ be continuous. Suppose that $\phi\in Lip(X)$ is maximized by a periodic orbit $\mclo$. Then for any $\ve>0$, there exists an open neighborhood $\mclv$ of $\phi-\ve\cdot\dist(\cdot,\mclo)$ in $Lip(X)$ such that each function in $\mclv$ is uniquely maximized by $\mclo$. 
\end{lem}

With the help of Lemma~\ref{lem:locking} and the fact that maximizing measures are generically unique, we can easily deduce Lemma~\ref{lem:reducible} below from Theorem~\ref{thm:Contreras}. Let us introduce some related terminology and notations first. 

\begin{defn}\label{defn:reducible}
Let $X$ be a compact metric space and let $T:X\to X$ be a Lipschitz map. Let $\mclp_T$ be the collection of $\phi\in Lip(X)$ such that $\phi$ can be maximized by some periodic orbit. We say that $(X,T)$ is {\bf reducible}, if there exists a dense subset $\mcld$ of $Lip(X)$ such that the following hold for any $\phi_*\in \mcld\setminus\mclp_T$. There exists an open neighborhood $\mclv$ of $\phi_*$ in $Lip(X)$ and a non-empty compact set $\Lambda\subset X$ with the following properties:
  \begin{itemize}
    \item $T\Lambda=\Lambda$, and $T:\Lambda\to \Lambda$ is open and expanding; 
    \item For any $\phi\in \mclv$, each maximizing measure of $\phi$ is supported in $\Lambda$.
  \end{itemize} 
\end{defn}

\begin{lem}\label{lem:reducible}
 If $(X,T)$ is reducible, then Lipschitz TPO holds for $(X,T)$.
\end{lem}

\begin{proof}
Let $\mcld$ be as given in Definition~\ref{defn:reducible} with required properties. Thanks to Lemma~\ref{lem:locking}, it suffices to show that $\mclp_T$ is dense in $Lip(X)$.   Since $\mcld$ is dense in $Lip(X)$, we only need to show that for any $\phi_*\in \mcld\setminus \mclp_T$ and any (sufficiently small) $\ve>0$, there exists a periodic orbit $\mclo$ such that for $\phi_{\ve,\mclo}:=\phi_*-\ve\cdot\dist(\cdot,\mclo)$, $\phi_{\ve,\mclo}$ is maximized by $\mclo$. Let $\mclv$ and $\Lambda$ be as described in Definition~\ref{defn:reducible}. There exists $\delta>0$ such that $\phi_{\ve,\mclo}\in \mclv$ for any $0<\ve\le \delta$ and any periodic orbit $\mclo$ of $(X,T)$.  
  Given $0<\ve\le \delta$, applying Theorem~\ref{thm:Contreras} to the subsystem $(\Lambda,T)$, we can find a periodic orbit $\mclo$ of $(\Lambda,T)$ such that the restriction of $\phi_{\ve,\mclo}$ to $\Lambda$ is maximized by $\mclo$ among invariant measures of $(\Lambda,T)$. Since $\phi_{\ve,\mclo}\in\mclv$, any maximizing measure of $\phi_{\ve,\mclo}$ for the system $(X,T)$ is supported in $\Lambda$. As a result, $\phi_{\ve,\mclo}$ is maximized by $\mclo$.
\end{proof}

The following well-known fact will be used occasionally without further explanation. 

\begin{lem}
  Let $T:X\to X$ be a continuous self-map on a compact metric space $X$. Then $\beta(\cdot)$ defined by \eqref{eq:max int} is continuous on $Lip(X)$ (or even on $C(X)$).
\end{lem}

\subsection{On the subordination principle in ergodic optimization}\label{sse:Morris}

In this subsection, we shall review a result of Morris \cite[Theorem~1]{Mor07}, namely Lemma~\ref{lem:Morris} and its direct corollary Lemma~\ref{lem:subord} stated below. For our purpose, the statements here is slightly stronger than that in \cite{Mor07}, so we shall provide a proof in \S~\ref{app:Morris} for completeness. Following the settings in \cite{Mor07}, at this moment $X$ is only assumed to be a topological space. For a Borel probability measure $\mu$ on $X$, recall $\supp\mu$ defined by \eqref{eq:supp}. Also recall that $\phi:X\to \mbbr$ is called {\bf upper semi-continuous}, if for any $t\in\mbbr$, the set $\{x\in X:\phi(x)<t\}$ is open.

\begin{lem}[\cite{Mor07}] \label{lem:Morris}
  Let $X$ be a topological space and let $T:X\to X$ be continuous. Let $\mu$ be a $T$-invariant Borel probability measure on $X$. Let $\phi:X\to\mbbr$ be upper semi-continuous. Suppose $\phi$ is $\mu$-integrable with $\int\phi\dif\mu=0$, and 
  \[
  \gamma:=\sup_{\substack{ n\ge 1 \\  x\in X }} \mcls_n\phi(x)<+\infty.
  \] 
Then we have\footnote{Note that from the definition of $\gamma$ and the assumption $\int\phi\dif\mu=0$ we always have $\gamma\ge 0$.}:
  \[
  \inf_{\substack{ n\ge 1 \\  x\in \supp\mu  }} \mcls_n\phi(x) \ge -\gamma.
  \]
\end{lem}

Now let $X$ be a compact metric space. For continuous $\phi:X\to\mbbr$, denote
\begin{equation}\label{eq:gamma}
  \gamma(\phi):=\sup_{\substack{ n\ge 1 \\  x\in X }} \big(
  \mcls_n\phi(x)-n\beta(\phi) \big) \in [0,+\infty].
\end{equation}
As an immediate corollary of Lemma~\ref{lem:Morris}, the following holds.

\begin{lem}[\cite{Mor07}]\label{lem:subord}
  Let $X$ be a compact metric space and let $T:X\to X$ be continuous. Let $\phi:X\to\mbbr$ be continuous with $\gamma(\phi)<+\infty$. Then for any maximizing measure $\mu$ of $\phi$, we have:
  \[
  \inf_{\substack{ n\ge 1 \\  x\in \supp\mu  }} \big[\mcls_n\phi(x) -n\beta(\phi)\big]\ge -\gamma(\phi).
  \]
As a corollary, the subordination principle holds for $\phi$: if $\nu\in\mclm_T$ satisfies $\supp\nu\subset\supp\mu$, then $\nu$ is also a maximizing measure of $\phi$.
\end{lem}

\subsection{On unimodal maps}\label{sse:unimodal}

In this subsection we introduce some terminologies about unimodal maps and summarize some well-known facts that will be used. 

Let $X=[a,b]\subset \mbbr$ be a closed interval. A continuous map $T:X\to X$ is called a {\bf unimodal map} with {\bf turning point} $c\in (a,b)$, if the following hold:
\begin{itemize}
  \item $T$ is not monotone on $[a,b]$ but strictly monotone on both $[a,c]$ and $[c,b]$;
  \item $T(\{a,b\})\subset\{a,b\}$.
\end{itemize}
\begin{rmk}
As in various studies on unimodal maps, we include the additional assumption $T(\{a,b\})\subset\{a,b\}$ in the definition above. This mild  technical assumption is inessential, but will help us to simplify some arguments and to improve readability (especially in case (ii) of the proof of Proposition~\ref{prop:subord}).   
\end{rmk}  

In this article, a unimodal map $T:X\to X$ is called {\bf piecewise expanding}, if it is a Lipschitz map and if there exists $C>0$ and $\lambda>1$ such that for any $n\ge 1$ and any subinterval $J$ of $X$ with $T^n:J\to X$ being monotone, we have $|T^nJ|\ge C\lambda^n |J|$. Note that since $T$ is Lipschitz, the last condition is equivalent to the following: $|(T^n)'(x)|\ge C\lambda^n$ for any $n\ge 1$ and for Lebesgue a.e. $x\in X$. 

For a $C^1$ unimodal map $T:X\to X$ with turning point $c$, we say that it satisfies the {\bf Collet-Eckmann condition}, if there exists $C>0$ and $\lambda>1$ such that $|(T^n)'(Tc)|\ge C\lambda^n$ for any $n\ge 1$. 

Following Nowicki and Sands \cite{NS98}, a unimodal map $T:[a,b]\to [a,b]$ with turning (or critical) point  $c$ and with $T(a)=a$ is called {\bf S-unimodal}, if the following hold:
\begin{itemize}
  \item $T$ is $C^2$ on $[a,b]$;
  \item $|T'|^{-1/2}$ is convex on both $(a,c)$ and $(c,b)$;
  \item the turning point $c$ is non-flat (we refer to \cite{NS98} for the precise meaning);
  \item $T'(a)>1$.
\end{itemize}

We shall not make use of the conditions imposed on unimodal maps introduced above directly. All we need are their implications summarized in Lemma~\ref{lem:finite renorm} and Lemma~\ref{lem:expshrink} below. Lemma~\ref{lem:finite renorm} concerns the topological dynamics, while Lemma~\ref{lem:expshrink} concerns expanding properties.


Before stating Lemma~\ref{lem:finite renorm}, we need to recall the following. Let $Y\subset \mbbr$ be a closed interval. A continuous map $T:Y\to Y$ is called {\bf locally eventually onto}, or {\bf l.e.o.} for short, if for any subinterval $I$ of $Y$ there exists $n\ge 1$ such that $T^n I=Y$. Now let $T:X\to X$ be a unimodal map with turning point $c$. It is well-known that, if  $T:X\to X$ is either a piecewise expanding unimodal map or an S-unimodal map satisfying the Collet-Eckmann condition, then it is {\bf finitely renormalizable}, which in particular implies the following.

\begin{lem}\label{lem:finite renorm}
  Let $T:X\to X$ be either a piecewise expanding unimodal map  or an S-unimodal map satisfying the Collet-Eckmann condition. 
  Then there exists $r\ge 1$ such that $c\in (T^rc, T^{2r}c)$ and for $Y:=[T^rc, T^{2r}c]$, $T^r Y=Y$; moreover, $T^r:Y\to Y$ is l.e.o.. As a corollary, we have:
\begin{itemize}
  \item periodic points of $T$ are dense in $Y$;
  \item for each $y\in Y$, $\cup_{n\ge 1}T^{-n}y$ is dense in $X$.
\end{itemize}
\end{lem}

\begin{rmk}
  We do not bother to discuss renormalization of unimodal maps in details, because all we need are given in Lemma~\ref{lem:finite renorm} above. We refer to \cite[\S~8]{NS98} for related concepts. 
  Let us only mention that, in Lemma~\ref{lem:finite renorm} above, $r=1$ corresponds to the case that $T$ is {\bf non-renormalizable}, while $r\ge 2$ implies that there exists a closed subinterval $I$ of $X$ such that $I\supset Y$ and $T^r:I\to I$ is a non-renormalizable unimodal map.   
\end{rmk}

%

%


In Lemma~\ref{lem:expshrink} below, when $T:X\to X$ is piecewise expanding, all the assertions  can be easily checked by definition. When $T:X\to X$ is S-unimodal and satisfies the Collet-Eckmann condition,  we refer to \cite{NP98,NS98}.

\begin{lem}\label{lem:expshrink}
  Let $T:X\to X$ be either a piecewise expanding  unimodal map or an S-unimodal map satisfying the Collet-Eckmann condition. Then the following hold.
\begin{itemize}
  \item For any non-empty compact subset $\Lambda\subset X$ with $c\notin \Lambda$ and $T\Lambda\subset\Lambda$, $T:\Lambda\to\Lambda$ is expanding. 
  \item $T:X\to X$ satisfies the {\bf exponential shrinking property}. That is to say, there exists $\ve>0$ and $\lambda\in(0,1)$ such that the following hold: if $I$ is a subinterval of $X$ with $|I|\le \ve$, then for any $n\ge 1$ and any connected component $J$ of $T^{-n}I$, we have $|J|\le \lambda^n$.  As a result, for any subinterval $J$ of $X$ and any $n\ge 1$, we have:
\begin{equation}\label{eq:shrink}
  |T^n J|\le \ve \implies \sum_{k=0}^{n-1} |T^k J| \le \frac{1}{1-\lambda}. 
\end{equation}
\end{itemize}
\end{lem}

We also need two elementary and well-known facts Lemma~\ref{lem:Markov} and Lemma~\ref{lem:open} given below. The former is to construct finite open covers with Markov property for invariant sets of a unimodal map, while the latter is to verify that the restriction of a unimodal map to some invariant set is an open map.

\begin{lem}\label{lem:Markov}
  Let $T:X\to X$ be a unimodal map with turning point $c$. Let $\mcli=\{(a_i,b_i)\subset X:1\le i\le n\}$ be a collection of open intervals with the following properties:
\begin{itemize}
  \item for each $I\in\mcli$, $c\notin I$;
  \item for each $1\le i\le n$, $Ta_i,Tb_i\notin \cup_{I\in\mcli}I$.
\end{itemize}
Then for any $I,J\in\mcli$ with $TI\cap J\ne\emptyset$, we have $TI\supset J$.
\end{lem}

\begin{proof}
  Denote $I=(a,b)$. Since $c\notin I$, $TI=(Ta,Tb)$. Combining this with $TI\cap J\ne\emptyset$ and $Ta,Tb\notin J$ we conclude that  $TI\supset J$.
\end{proof}

\begin{lem}\label{lem:open}
  Let $X$ be a topological space and let $T:X\to X$ an arbitrary map. Let $U\subset X$ be open and suppose $\Lambda:=\cap_{n=0}^{\infty} T^{-n}U$ is nonempty. If $T:U\to X$
  is an open map, then $T:\Lambda\to \Lambda$ is also an open map. As a corollary, if $T:X\to X$ is a unimodal map with turning point $c$ and if $c\notin U$, then $T:\Lambda\to \Lambda$ is an open map.
\end{lem}

\begin{proof}
  By the definition of $\Lambda$, $U$ is an open neighborhood of $\Lambda$ in $X$ and $T\Lambda\subset \Lambda=U\cap T^{-1}\Lambda$. Given $x\in\Lambda$, since $T:U\to X$ is an open map and since $x,Tx\in U$, there exists an open neighborhood $V$ of $x$ in $X$ and an open neighborhood $V'$ of $Tx$ in $X$, such that $V,V'\subset U$ and $TV\supset V'$. Since $\Lambda=U\cap T^{-1}\Lambda$ and since $V\subset U$, $T(V\cap\Lambda)=T(V\cap T^{-1}\Lambda)=TV\cap \Lambda\supset V'\cap \Lambda$, which implies that  $T:\Lambda\to\Lambda$  is an open map. 
\end{proof}

\section{Proof of the Main results}\label{se:unimodal}

In this section, we shall prove Theorem~\ref{thm:unimodal} in \S~\ref{sse:proof unimodal} and then deduce Theorem~\ref{thm:exclu} from it in \S~\ref{sse:proof exclu}. For preparation, we need to prove Proposition~\ref{prop:subord} first.

\subsection{Establishing subordination}\label{sse:subord}

This subsection is devoted to proving the following. Recall $\gamma(\phi)$ defined by \eqref{eq:gamma}. 

\begin{prop}\label{prop:subord}
  Let $T:X\to X$ be either a piecewise expanding unimodal map or an S-unimodal map satisfying the Collet-Eckmann condition. Then there exists a constant $C>0$ such that  $\gamma(\phi)\le C\lip(\phi)$ holds for any $\phi\in Lip(X)$.
\end{prop}

\begin{proof}
Given $\phi\in Lip(X)$, $x\in X$ and $n\ge 1$, it suffices to show that $\mcls_n\phi(x)-n\beta(\phi)\le C\lip(\phi)$ for some $C>0$ independent of $\phi,x,n$. To begin with, note that 
\[
|\phi(z)-\beta(\phi)|\le  |X|\cdot \lip(\phi) , \quad \forall z\in X,
\]
which will be used several times.  Let $r\ge 1$ be as given in Lemma~\ref{lem:finite renorm} and denote $Y=[T^rc,T^{2r}c]$. Since $Y$ is $T^r$-invariant, we only need to deal with the following two cases: 
\begin{itemize}
  \item [(i)] $x\in Y$;
  \item [(ii)] $x,Tx,\cdots,T^{n-1}x\notin Y$.  
\end{itemize}     

{\bf Case (i)}. In this case we shall compare $\mcls_n\phi(x)$ with $\mcls_n\phi(y)$ for some periodic point $y$. Let $\ve>0$ and $\lambda\in(0,1)$ be as given in Lemma~\ref{lem:expshrink}, and we may further suppose that $|Y|\ge \ve$. Since $T^r: Y\to Y$ is l.e.o., there exists $m\ge 1$, independent of $\phi,x,n$, such that if $J$ is a subinterval of $Y$ with  $\max_{0\le k<r}|T^k J|\ge \ve$, then $T^s J=Y$ for any $s\ge m$ with $r\mid s$.  Now let $I$ be a closed subinterval of $Y$ with $x\in  I$ and $\max_{0\le k<n}|T^kI|=\ve$, whose existence follows easily from the continuity of $\max_{0\le k<n}|T^kI|$ as a function of the endpoints of $I$ together with $|Y|\ge \ve$. Let $p\ge n+m$ be the minimal integer with $r\mid p$. Then $T^p I=Y$ and hence there exists a fixed point $y\in I$ of $T^p$. Now we have the following estimates: 
\[
\sum_{0\le k<n} [\phi(T^k x)-\phi(T^k y)]\le \lip(\phi) \cdot \sum_{0\le k<n} |T^k I|  \le \frac{\lip(\phi)}{1-\lambda}, 
\]
where the latter inequality is due to \eqref{eq:shrink}, and

\[
\begin{split}
 \sum_{0\le k<n} [\phi(T^k y)-\beta(\phi)]    & = \sum_{0\le k<p} [\phi(T^k y)-\beta(\phi)] - \sum_{n\le k<p} [\phi(T^k y)-\beta(\phi)] \\
     & \le 0 +  (m+r)|X|\cdot \lip(\phi).
\end{split}
\]
Combining the two displayed lines above, the proof in this case is done.

{\bf Case (ii)}. In this case we shall split the orbit $x,Tx,\cdots,T^{n-1}x$ 
into several pieces (the number of pieces will be bounded by a constant independent of $\phi,x,n$), and compare the Birkhoff sum of $\phi$ along each piece with 
that along some period orbit. Let $z\in Y$ be a periodic point such that $\cup_{k\ge 0}T^{-k} z$ is dense in $X$, whose existence is guaranteed by Lemma~\ref{lem:finite renorm}. Then we can find $m\ge 1$, independent of $\phi,x,n$, such that for $E_m:=\cup_{0\le k < m}T^{-k} z$ the following properties hold.
\begin{itemize}
  \item $TE_m\subset E_m$. 
  \item There exist  $u,v\in E_m$ such that $c\in [u,v]\subset Y$.
  \item Let $\mcli$ denote the collection of connected components of $X\setminus ([u,v]\cup E_m)$. For each $I\in\mcli$, $|I|\le\ve$. 
\end{itemize} 
From the properties above we know that $x,Tx,\cdots,T^{n-1}x\in \cup_{I\in\mcli} \overline{I}$. Moreover, together with the fact that  $T(\{a,b\})\subset \{a,b\}$ for $X=[a,b]$, by Lemma~\ref{lem:Markov}, we have:  for $I,J\in\mcli$, if $TI\cap J\ne\emptyset$, then $TI\supset J$.

Let us pick up pairwise distinct $I_0,\cdots,I_l\in \mcli$ and associated times $0=s_0\le t_0<\cdots<s_l\le t_l=n-1$ by induction as follows, where $0\le l<\#\mcli$. To begin with, let $s_0=0$ and let $I_0\in\mcli$ be such that $x\in \overline{I_0}$. Once $s_k$ and $I_k$ are chosen, let $t_k<n$ be the maximal integer such that $T^{t_k}x\in \overline{I_k}$. If $t_k=n-1$, then the induction is ended and $l=k$; otherwise, let $s_{k+1}=t_k+1$ and choose $I_{k+1}\in\mcli$ such that $T^{s_{k+1}}x\in \overline{I_{k+1}}$. 

According to the choice of $I_0,\cdots,I_l$ and associated $s_k,t_k$ as above, when $s_k<t_k$, there exists a closed subinterval $J_k$ of $\overline{I_k}$ such that $T^{s_k}x\in J_k$ and $T^{t_k-s_k}$ maps $J_k$ onto $\overline{I_k}$. In this situation,  there exists a fixed point $y_k\in J_k$ of $T^{t_k-s_k}$ and hence 
\[
\sum_{j=s_k}^{t_k-1}[\phi(T^j x)-\beta(\phi)] \le \sum_{i=0}^{t_k-s_k-1}[\phi(T^{s_k+i} x)- \phi(T^i y_k)] \le \frac{\lip(\phi)}{1-\lambda},
\]
where the latter inequality is due to \eqref{eq:shrink}. It follows that 

\[
\begin{split}
 \sum_{j=0}^{n-1} [\phi(T^j x)-\beta(\phi)]    & =\sum_{k=0}^l [\phi(T^{t_k} x)-\beta(\phi)] + \sum_{\substack{0\le k\le l \\ s_k<t_k}} \sum_{j=s_k}^{t_k-1}[\phi(T^j x)-\beta(\phi)] \\
     & \le (l+1)\cdot \left(|X|\cdot \lip(\phi) + \frac{\lip(\phi)}{1-\lambda}\right).
\end{split}
\]
Since $l+1\le \#\mcli\le \# E_m +1\le 2^{m}$, the proof in this case is done.
\end{proof}

\subsection{Proof of Theorem~\ref{thm:unimodal}} \label{sse:proof unimodal}

In this subsection we prove Theorem~\ref{thm:unimodal}. Due to Lemma~\ref{lem:reducible}, we only need to prove the following.

\begin{prop}\label{prop:reducible}
  $(X,T)$ given in Theorem~\ref{thm:unimodal} is reducible.
\end{prop}

The rest of this subsection is devoted to the proof of  Proposition~\ref{prop:reducible}. We need Lemma~\ref{lem:admissible Markov} and Lemma~\ref{lem:u.s.c.} below for preparation. To state Lemma~\ref{lem:admissible Markov}, let us introduce some terminology for convenience. 
\begin{defn}
  Let $T:X\to X$ be a Lipschitz unimodal map. Let $K$ be a non-empty compact subset of $X$ with $TK=K$. Let $\mcli$ be a collection of finitely many open subintervals in $\mbbr$. We say that 
  $(K,\mcli)$ is an {\bf admissible pair}, if $K\subset \cup_{I\in \mcli}I\subset X$, and for $\Lambda:=\cap_{n\ge 0} T^{-n}(\cup_{I\in \mcli} I)$,  $\Lambda$ is compact and $T:\Lambda\to\Lambda$ is surjective, open and expanding.
\end{defn}

\begin{lem}\label{lem:admissible Markov}
Let $X=[a,b]$ and let $T:X\to X$ be either a piecewise expanding unimodal map or  an S-unimodal map satisfying the Collet-Eckmann condition with turning point $c$. Let $K\subset (a,b)$ be a non-empty compact set with $TK=K$. If $c\notin K$, then there exists $\mcli$ such that $(K,\mcli)$ is an  admissible pair.
\end{lem}

\begin{proof}
 Since $K$ is compact and $c\notin K$, by Lemma~\ref{lem:finite renorm}, there exists a periodic point $z\notin K$ such that $\cup_{k\ge 0}T^{-k}z$ is dense in $X$. Denote the orbit of $z$ by $z_1,\cdots,z_p$. Then $(\cup_{k\ge 0}T^{-k}z)\cap K=\emptyset$ and in particular $z_1,\cdots,z_p\notin K$.

 By choosing $m$ large, for $E_m:=\cup_{0\le k<m}T^{-k}z$,  $TE_m\subset E_m$, $E_m\cap K=\emptyset$ and the following hold: there exist  $a_1,b_1,a_2,b_2\in E_m$ and $x_i,y_i\in E_m, 1\le i\le p$, such that $c\in (a_1,b_1)$, $[a_2,b_2]\subset (a,b)$, $z_i\in (x_i,y_i), 1\le i\le p$ and $K\subset (a_2,b_2)\setminus F$ for $F:= [a_1,b_1]\cup (\cup_{i=1}^p[x_i,y_i])$. Let $\mcli_0$ be the collection of connected components of $(a_2,b_2)\setminus (E_m\cup F)$ and let $\mcli=\{I\in\mcli_0: I\cap K\ne\emptyset\}$. 
 
 By the definition of $\mcli$, for $U:=\cup_{I\in\mcli} I$, $U\supset K$. Moreover, for any $J\in\mcli$, under iteration of $T$, both endpoints of $J$ will be eventually mapped to the orbit of $z$, which does not intersect $\overline{U}=\cup_{I\in\mcli}\overline{I}$. Therefore, for $\Lambda:=\cap_{n\ge 0} T^{-n} U$, $\Lambda=\cap_{n\ge 0} T^{-n}\overline{U}$ is compact.
 
 To show $T:\Lambda\to\Lambda$ is surjective, it suffices to show that $T U\supset U$. First note that by Lemma~\ref{lem:Markov}, for any $I,J\in\mcli$, $TI\cap J\ne\emptyset$ implies that $TI\supset J$. Now given $J\in\mcli$, since $J\cap K\ne\emptyset$, $TK=K$ and $U=\cup_{I\in\mcli} I\supset K$, there exists $I\in \mcli$ such that $TI\cap J\ne\emptyset$, which implies that $TI\supset J$. The conclusion follows.
 

Finally, since $c\notin U$ and since $\Lambda=\cap_{n\ge 0} T^{-n} U$, from Lemma~\ref{lem:open} we know that $T:\Lambda\to\Lambda$ is an open map; from our assumption on $T$ and Lemma~\ref{lem:expshrink} we know that $T$ is expanding on $\Lambda$.
\end{proof}

Given $\phi_*\in Lip(X)$,  according to Proposition~\ref{prop:subord}, there exists an open neighborhood $\mclv_*$ of $\phi_*$ in $Lip(X)$ and $C_*>0$ such that for any $\phi\in \mclv_*$ we have $\gamma(\phi)\le C_*$. For each $\phi\in\mclv_*$, define
\[
\whK_\phi:=\{x\in X: |\mcls_k \phi(T^l x) - k\beta(\phi)|\le C_*,\forall k\ge 1,l\ge 0\}\ \ \mbox{and}\ \ K_\phi:=\cap_{m=0}^\infty T^m \whK_\phi.
\]

\begin{lem}\label{lem:u.s.c.}
 The following hold.
 \begin{itemize}
   \item [(1)] For each $\phi\in \mclv_*$,  $K_\phi$ is compact and  $T K_\phi=K_\phi$; moreover, $\mu\in\mclm_T$ maximizes $\phi$ iff $\supp\mu\subset K_\phi$.
   \item [(2)]  The map $\phi\mapsto K_\phi$ defined on $\mclv_*$ above is upper semi-continuous in the following sense: given $\phi\in \mclv_*$ and an open neighborhood $V$ of $K_\phi$ in $X$, there exists an open subset $\mclv$ of $\mclv_*$ such that $\phi\in\mclv$ and $K_\varphi\subset V$ for any $\varphi\in \mclv$.
 \end{itemize}
\end{lem}

\begin{proof}
(1).  By definition,  $K_\phi$ is compact, $T K_\phi=K_\phi$ and for $\mu\in\mclm_T$, if $\supp\mu\subset K_\phi$, then $\mu$ maximizes $\phi$. On the other hand, from $\gamma(\phi)\le C_*$ and Lemma~\ref{lem:subord} we know that $K_\phi$ contains the support of each maximizing measure of $\phi$. 

(2).  Given $\phi\in\mclv_*$, it suffices to show that the following hold: for $\phi_n\in \mclv_*$ with $\lim_n \phi_n=\phi$ in $Lip(X)$ and for $x_n\in K_{\phi_n}$ with $\lim_n x_n=x\in X$, we always have $x\in K_\phi$. Since $K_\phi=\cap_{m\ge 0}T^m\whK_\phi$, fixing an arbitrary $m\ge 0$, we have to show that $x\in T^m\whK_\phi$.  For each $n$, since $x_n\in K_{\phi_n}\subset T^m\whK_{\phi_n}$, there exists $y_n\in T^{-m} x_n\cap \whK_{\phi_n}$. By taking a subsequence if necessary, we may assume that $\lim_n y_n=y$ for some $y\in X$. Then $x=T^m y$, so we only need to show that $y\in \whK_\phi$. To this end, by the definition of $\whK_\phi$, fixing arbitrary $k\ge 1,l\ge 0$, it suffices to show that $|\mcls_k\phi(T^ly)-k\beta(\phi)|\le C_*$. This follows immediately from $\lim_n \phi_n=\phi$, the continuity of $\beta(\cdot)$, $\lim_n y_n=y$ and $|\mcls_k\phi_n(T^ly_n)-k\beta(\phi_n)|\le C_*$. 
\end{proof}

Now we are ready to prove  Proposition~\ref{prop:reducible}.

\begin{proof}[Proof of Proposition~\ref{prop:reducible}]
Given $\phi_*\in Lip(X)$, let us follow the notations $C_*$, $\mclv_*$ associated to $\phi_*$ and $K_\phi$ for $\phi\in\mclv_*$ introduced previously. According to Definition~\ref{defn:reducible}, it suffices to find a dense subset $\mcld$ of $\mclv_*$ such that for any $\varphi_*\in \mcld\setminus\mclp_T$, there exists an open neighborhood $\mclv$ of $\varphi_*$ and $\Lambda\subset X$ with the desired properties stated in Definition~\ref{defn:reducible}. To this end, denote $X=[a,b]$, let $c$ be the turning point of $(X,T)$ and let 
\[
\mcld:=\{ \phi\in \mclv_*: \mbox{either $c\notin K_{\phi}$ or $\phi\in\mclp_T$} \}.
\]
We claim that $\mcld$ is dense in $\mclv_*$. To prove this claim, let us deal with the following three cases separately.

{\bf Case (i)}  $(\omega(c),T)$ is not uniquely ergodic. Let $\mclu_T$ be the collection of $\phi\in Lip(X)$ such that $\phi$ has a unique maximizing measure. 
As is well-known, $\mclu_T$ is dense in $Lip(X)$ (see, for example, \cite[Theorem~2.4]{Jen19}). On the other hand, since $(\omega(c),T)$ is not uniquely ergodic, by assertion~(1) in Lemma~\ref{lem:u.s.c.}, for each $\phi\in\mclu_T\cap\mclv_*$, $c\notin K_\phi$. It follows that $\mcld$ is dense in $\mclv_*$.

{\bf Case (ii)} $\omega(c)$ contains a periodic orbit. In this case $\mcld=\mclv_*$ by definition and assertion~(1) in Lemma~\ref{lem:u.s.c.}.

{\bf Case (iii)} $c\notin \omega(c)$.  Let $\phi\in \mclv_*$ and let $\mclw\subset\mclv_*$ be an open neighborhood of $\phi$. We have to show $\mclw\cap \mcld\ne\emptyset$.
  We may assume that $c\in K_{\phi}$ and $\phi\notin\mclp_T$, because otherwise there is nothing to prove. These assumptions together with assertion~(1) in Lemma~\ref{lem:u.s.c.} imply that  $\omega(c)$ contains no periodic point; in particular, $\omega(c)\subset (a,b)$. 
  Then applying Lemma~\ref{lem:admissible Markov} to $K=\omega(c)$,  we can find a compact $\Lambda\supset \omega(c)$ such that $T:\Lambda\to\Lambda$ is surjective, open, and expanding. To proceed, let $\delta>0$ be such that  $\phi_{\ve,\mclo}:=\phi-\ve\cdot\dist(\cdot,\mclo)\in\mclw$ for any $0<\ve\le\delta$ and any periodic orbit $\mclo$.
  Now applying Theorem~\ref{thm:Contreras} to the subsystem $(\Lambda,T)$ and the restriction of $\phi$ to $\Lambda$, for any $0<\ve\le \delta$ there exists a periodic orbit 
  $\mclo_\ve\subset \Lambda$ such that restricted to the subsystem $(\Lambda,T)$, $\varphi_\ve:=\phi_{\ve,\mclo_{\ve}}\in\mclw$ is uniquely maximized by $\mclo_\ve$. Since $\omega(c)$ is a compact invariant subset of $\Lambda$ and since it contains no periodic point, it follows that any weak-* limit point of $\frac{1}{n}\sum_{k=0}^{n-1}\delta_{T^k c}$ 
 cannot be a maximizing measure of $\varphi_\ve$, and hence $c\notin K_{\varphi_\ve}$. The conclusion in this case follows.

 Now given $\varphi_*\in \mcld\setminus\mclp_T$, let us complete the proof. By definition, $c\notin K_{\varphi_*}$ and $K_{\varphi_*}\subset(a,b)$. Therefore, we can apply  Lemma~\ref{lem:admissible Markov} to $K=K_{\varphi_*}$ and then find $\mcli$ such that $(K_{\varphi_*},\mcli)$ is an admissible pair. In particular, $\cup_{I\in\mcli} I$ is an open neighborhood of $K_{\varphi_*}$. Then by  assertion~(2) in Lemma~\ref{lem:u.s.c.}, there exists an open neighborhood $\mclv\subset\mclv_*$ of $\varphi_*$ in $Lip(X)$, such that $K_\phi\subset\cup_{I\in\mcli} I$ for any $\phi\in\mclv$. Combining this with $TK_\phi=K_\phi$ we obtain that for $\Lambda=\cap_{n\ge 0} T^{-n}(\cup_{I\in \mcli} I)$, $K_\phi\subset\Lambda, \forall \phi\in\mclv$. Also recall that $T:\Lambda\to\Lambda$ is surjective, open and expanding. The conclusion follows.
\end{proof}

%

\subsection{Proof of Theorem~\ref{thm:exclu}}\label{sse:proof exclu} For either a quadratic map $Q_a$ or a tent map $T_a$, let $c$ denote its turning point below. Let us consider the quadratic case first. For a quadratic map $Q_a:[0,1]\to [0,1]$, a periodic orbit $\mclo$ of $Q_a$ with period $p$ is called {\bf attracting}, if $|(Q_a^p)'(y)|<1$ for  $y\in\mclo$.

\begin{prop}\label{prop:attracting}
Let $a\in (a_*,4]$ and suppose that $Q_a$ has an attracting periodic orbit $\mclo$. Then Lipschitz TPO holds for $(X,Q_a)$.
\end{prop}

\begin{proof}
By Lemma~\ref{lem:reducible} again, let us show that $(X,Q_a)$ is reducible. To this end, given $\phi_*\in Lip(X)\setminus\mclp_T$, 
it suffices to find $\mclv$ and $\Lambda$ as required in Definition~\ref{defn:reducible}. $\mclv$ is determined as follows.  Since $\phi_*\notin \mclp_T$, $\phi_*$ cannot be maximized by either $\{0\}$ or $\mclo$. Then by continuity, there exists an open neighborhood $\mclv$ of $\phi_*$ such that the same holds for $\phi\in\mclv$ instead of $\phi_*$. 

Now let us define $\Lambda$. To begin with, let $X:=[Q_a^2(c),Q_a(c)]$. Note that $a_*>1+\sqrt{5}$ and direct calculations show that $c\in X$ and $Q_a X=X$ when $a\in (1+\sqrt{5},4]$. Then let $B=\{x\in [0,1]:\omega(x)=\mclo\}$ be the basin of attraction of $\mclo$. It is well-known that $B$ is an open subset of $(0,1)$, $\mclo\subset B$, $c\in B$, and $Q_a^{-1}B =B$. Let $\Lambda:=X\setminus B$. By definition, $\Lambda$ is compact. Since $Q_a:X\to X$ is surjective and since $Q_a^{-1}B=B$, $Q_a\Lambda=\Lambda$. It is well-known that $Q_a:\Lambda\to\Lambda$ is expanding. 

To see $Q_a:\Lambda\to\Lambda$ is an open map, let $U\subset (Q_a^2(c),Q_a(c))\setminus\{c\}$ be an open neighborhood of $\Lambda$ with $\mclo\cap \overline{U}=\emptyset$. From the choice of $U$ and the definition of $\Lambda$ we know that $\Lambda=\cap_{n=0}^{\infty}Q_a^{-n}U$. Then the conclusion follows from Lemma~\ref{lem:open}.

Finally, note that $\cap_{n=0}^\infty Q_a^n(0,1)\subset X$, so for any $x\in (0,1)$, either $\omega(x)=\mclo$ or $\omega(x)\subset\Lambda$. For each $\phi\in\mclv$, since $\phi$ is not maximized by $\{0\}$ or $\mclo$, it follows that the support of any maximizing measure of $\phi$ is contained in $\Lambda$. The proof is done.
\end{proof}

\begin{proof}[Proof of Theorem~\ref{thm:exclu}~(1)]
It is well-known that $Q_a:[0,1]\to [0,1]$ is an S-unimodal map. Let $A$ be the collection of $a\in [a_*,4]$ such that $Q_a$ has an attracting periodic point. Due to Proposition~\ref{prop:attracting}, we only need to consider $a\in [a_*,4]\setminus A$. According to Avila and Moreira \cite{AM05p,AM05a}, for Lebesgue almost every $a\in [a_*,4]\setminus A$, $Q_a$ satisfies the Collet-Eckmann condition (see \cite[Theorem~A]{AM05a}) and $\omega(c)$ is the union of finitely many closed intervals (see \cite[Theorem~2]{AM05p}).  Then the conclusion follows from Theorem~\ref{thm:unimodal}. 
\end{proof}

\begin{proof}[Proof of Theorem~\ref{thm:exclu}~(2)]
By definition, for each $a\in (1,2]$, $T_a:[0,2]\to [0,2]$ is a piecewise expanding unimodal map. On the other hand, according to Brucks and Misiurewicz \cite{BM96} or Bruin \cite{Bru98}, for Lebesgue a.e. $a\in (1,2]$, $\omega(c)$ is the union of finitely many closed intervals. Then the conclusion follows from Theorem~\ref{thm:unimodal}. 
\end{proof}

\section{On non-uniformly expanding circle covering maps}\label{se:circle}

As a supplement of Theorem~\ref{thm:unimodal}, in this section we show that Lipschitz TPO holds for a different kind of one-dimensional non-uniformly expanding maps: circle covering maps that are expanding away from an indifferent fixed point. The main result in this section is the following.

\begin{thm}\label{thm:circle}
  Let $X=\mbbr/\mbbz$ be the unit circle. Let $T:X\to X$ be a covering map with a fixed point $c\in X$. Assume that $T$ is a Lipschitz map, and for any open neighborhood $V$ of $c$, there exists $C>1$ such that $|T'(x)|\ge C$ for Lebesgue a.e. $x\in X\setminus V$. Then Lipschitz TPO holds for $(X,T)$.
\end{thm}

\begin{rmk}
  The assumption on $(X,T)$ in Theorem~\ref{thm:circle} implies that $|\deg T|\ge 2$, where $\deg T$ denotes the degree of $T$ as a covering map of the unit circle. 
\end{rmk}

\subsection{Proof of Theorem~\ref{thm:circle}}

By Lemma~\ref{lem:reducible} again, we only need to show that $(X,T)$ is reducible, which can be easily deduced from the following.

\begin{prop}\label{prop:sub}
 Let $(X,T)$ be as given in Theorem~\ref{thm:circle} and let  $\phi\in Lip(X)$. If $\phi(c)<\beta(\phi)$, then 
\begin{equation}\label{eq:sub}
    \psi(x):=\sup_{n\ge 1} \max_{y\in T^{-n} x} [\mcls_n \phi(y)-n\beta(\phi)], \quad \forall\,x\in X
\end{equation}
is a Lipschitz sub-action of $\phi$. Moreover, for any $\phi_*\in Lip(X)$ with $\phi_*(c)<\beta(\phi_*)$, there exists an open neighborhood $\mclv$ of $\phi_*$ in $Lip(X)$ and $L>0$ such that for any $\phi\in\mclv$, the associated $\psi$ to $\phi$ defined by \eqref{eq:sub} is real-valued and $\lip(\psi)\le L$. 
\end{prop}

\begin{rmk}
   $\psi$ defined by \eqref{eq:sub} is a well-known natural candidate of sub-actions of $\phi$; see, for example, \cite{CLT01,Jen19}. As indicated by the proof of \cite[Theorem~2]{Mor09} or by \cite[Theorem~1]{GI20}, the assumption $\phi(c)<\beta(\phi)$ in Proposition~\ref{prop:sub} is necessary. We refer to \cite{Mor09,Jen19,GI20} and references therein for more discussions on existence and regularity of sub-actions for (non-uniformly expanding) circle maps.
\end{rmk}

We postpone the proof of Proposition~\ref{prop:sub} to the next subsection and complete the proof of Theorem~\ref{thm:circle} first. To begin with, let us define a nested sequence of open neighborhoods $V_n$ of $c$ that will be used in both proofs. For each $n\ge 1$,  let $a_n,b_n\in T^{-n}c$ be the two points in $(T^{-n}c)\setminus \{c\}$ that are adjacent to $c$ (since $\# (T^{-n}c) =|\deg T|^n\ge 2^n$,  $a_n\ne b_n$ unless both $n=1$ and $|\deg T|=2$ hold), and let  $V_n$ be the connected component of $X\setminus\{a_n,b_n\}$ containing $c$.

\begin{proof}[Proof of Theorem~\ref{thm:circle}]
  According to Lemma~\ref{lem:reducible}, we only need to show that $(X,T)$ is reducible. Recall Definition~\ref{defn:reducible}. Let $\phi_* \in Lip(X)\setminus\mclp_T$, so that $\phi_*(c)<\beta(\phi_*)$. Then by Proposition~\ref{prop:sub}, there exists an open neighborhood $\mclw$ of $\phi_*$ in $Lip(X)$ and $L>0$ such that for any $\phi\in\mclw$, $\psi$ defined by \eqref{eq:sub} is a Lipschitz sub-action of $\phi$ with $\lip(\psi)\le L$. For $\phi\in \mclw$ and its associated $\psi$, let
  \[
  Z_\phi:=\{x\in X: \phi(x)+\psi(x)=\psi(Tx)+\beta(\phi)\},
  \]
and note that $c\notin  Z_\phi$ provided that $\phi(c)<\beta(\phi)$. Then by continuity, there exists $m\ge 1$ and an open neighborhood $\mclv\subset \mclw$ of $\phi_*$ such that for $\phi\in \mclv$, $Z_\phi\cap V_m=\emptyset$. Also recall the well-known fact that the support of any maximizing measure of $\phi$ is contained in $Z_\phi$. It follows that $\mclv$ and $\Lambda:=\cap_{n\ge 0}T^{-n} (X\setminus V_m)$ are as required in Definition~\ref{defn:reducible}, which completes the proof.  
\end{proof}

\subsection{On the existence of Lipschitz sub-actions}

This subsection is devoted to the proof of Proposition~\ref{prop:sub}. 

To begin with, let us summarize some basic facts of the dynamics of $(X,T)$ that will be used. Firstly, for each $n\ge 1$, there exists $\lambda_n\in (0,1)$ such that $|T'(x)|\ge \lambda_n^{-1}$ for Lebesgue a.e. $x\in X\setminus V_n$. Then for any arc $J$ in $X$ and any $l,s\ge 1$, we have:
\begin{equation}\label{eq:exp contract}
 T^l J\ne X ~\mbox{and}~ \#\{0\le k<l: T^k J \cap V_s =\emptyset \} \ge r \implies |J|\le \lambda_s^r\cdot|T^l J|.
\end{equation}
This is because the map $T^l:J\to T^l J$ is bi-Lipschitz  and $|(T^l)'(x)|\ge \lambda_s^{-r}$ holds for Lebesgue a.e. $x\in J$. Secondly, for each $n\ge 1$, let $\mcli_n$ denote the collection of connected components of $X\setminus (T^{-n}c)$. Then for any $I\in \mcli_n$, there exists a unique fixed point $y$ of $T^n$ with $y\in \overline{I}$ (note that every periodic point of $T$ is of this form); moreover, $y\in I$ iff $y\ne c$ iff $I\cap V_n=\emptyset$.

Now for $\phi_*$ given in the statement of Proposition~\ref{prop:sub}, let us fix some notations and constants, and specify the choice of $\mclv$. Denote $6\ve:=\beta(\phi_*)-\phi_*(c)>0$. Since invariant measures of $T$ can be approximated by measures supported on periodic orbits in the weak-* topology, there exists a periodic point $y_*$ with period $p$ such that $\frac{1}{p} \mcls_p \phi_*(y_*) \ge \beta(\phi_*)-\ve$. Note that $y_*\ne c$, and hence there exists a unique $I_*\in \mcli_p$ with $y_*\in I_*$. By continuity, there exists an open neighborhood $\mclv$ of $\phi_*$ and constants $C_*,L_*>0$ such that the following hold for any $\phi\in \mclv$:
\begin{itemize}
  \item $\beta(\phi)-\phi(c)\ge 5\ve$;
  \item $\frac{1}{p} \mcls_p \phi(y_*) \ge \beta(\phi)-2\ve$;
  \item $|\phi-\beta(\phi)|\le C_*$ on $X$; 
  \item $\lip(\phi)\le L_*$.
\end{itemize} 

From now on let us fix an arbitrary $\phi\in \mclv$ and for convenience,  denote
\[
\overline{\phi}:=\phi-\beta(\phi), \qquad  \psi_n(x):=\max_{y\in T^{-n} x} \mcls_n \overline{\phi}(y), \quad \forall\,x\in X,\forall\, n\ge 1.
\]
To complete the proof, we only need to find constants $C,L>0$ independent of $\phi\in\mclv$, such that 
$\psi_n\le C$ and $\lip(\psi_n)\le L$ for any $n\ge 1$. To proceed, we need the following two claims. 

\begin{clm}\label{clm:freq}
  There exist integers $s,t\ge 1$ and $\eta\in(0,1)$ such that the following hold for any $\phi\in\mclv$. Given $x\in X$ and $n\ge t$, if $y\in T^{-n}x$ satisfies $\psi_n(x)=\mcls_n \overline{\phi}(y)$, then
\[
\#\{0\le k< n: T^k y \notin \overline{V_s} \} \ge \eta n. 
\]

\end{clm}

\begin{proof}
Let $t$ be the minimal positive integer such that 
\[
t\ve\ge p\left(\frac{L_*}{1-\lambda_p} + C_*\right) 
\]
holds. Fixing $x\in X$ and $n\ge t$ below, let us show that $\psi_n(x)\ge -3n\ve $ first.  Let $q\ge 1$ be the unique integer determined by $pq\le n<p(q+1)$. Since $T^p\overline{I_*}=X$, there exists $y'\in T^{-pq}x$ with $y'\in J:=\cap_{j=0}^{q-1}T^{-jp}\overline{I_*}$. Also recall that due to the choice of $I_*$, $\overline{I_*}\cap V_p=\emptyset$. Then from $\lip(\phi)\le L_*$, $y',y_*\in J$, $T^{jp}J\cap V_p=\emptyset, 0\le j<q$ and \eqref{eq:exp contract} we have:
\[
|\mcls_{pq}\phi(y')-\mcls_{pq}\phi(y_*)|\le L_*\cdot \sum_{0\le j<q}\sum_{jp\le k<(j+1)p} |T^k J| \le  L_*\cdot \sum_{0\le j<q} p\lambda_p^{q-j-1} \le \frac{pL_*}{1-\lambda_p}.
\]
Therefore,
\[
\mcls_{pq} \overline{\phi}(y') \ge \mcls_{pq} \overline{\phi}(y_*) -\frac{pL_*}{1-\lambda_p} \ge -2pq\ve -\frac{pL_*}{1-\lambda_p}.
\]
Take an arbitrary $y''\in T^{-(n-pq)}y'$, so that $T^n y''=x$.  It follows that 
\[
\psi_n(x) \ge \mcls_n\overline{\phi}(y'')\ge \mcls_{pq}\overline{\phi}(y') - p C_* \ge -2pq\ve -\frac{pL_*}{1-\lambda_p}- p C_*\ge -3n\ve,
\]
where the last inequality is due to $n\ge pq$, $n\ge t$ and the choice of $t$.
 
On the other hand, since $\lip(\phi)\le L_*$, there exists $s\ge 1$ independent of $\phi$ such that for $z\in \overline{V_s}$, $\phi(z)\le \phi(c)+\ve$ and hence $\overline{\phi}(z)\le -4\ve$ . Now assume that $y\in T^{-n}x$ satisfies $\psi_n(x) = \mcls_n \overline{\phi}(y)$. Then for $m:=\{0\le k< n: T^k y \notin \overline{V_s}\}$,
\[
\psi_n(x) = \mcls_n \overline{\phi}(y)  \le -4(n-m)\ve + mC_*.
\]
Combining this with $\psi_n(x)\ge -3n\ve$, the conclusion follows by taking $\eta=\frac{\ve}{C_*+4\ve}$.
\end{proof}

From now on let us fix the constants $s,t$ and $\eta$ introduced in Claim~\ref{clm:freq}.

\begin{clm}\label{clm:lip diff}
Let $I\in\mcli_s$ and let $x\in\overline{I}$. Given $n\ge 1$, let $J\in\mcli_{n+s}$ and $y\in T^{-n}x$ satisfy that $T^nJ=I$, $y\in\overline{J}$ and $\psi_n(x)=\mcls_n\overline{\phi}(y)$. Then for $x'\in\overline{I}$ and $y'\in T^{-n} x'\cap \overline{J}$, we have:
  \[
  \psi_n(x)=\mcls_n\overline{\phi}(y)\le \mcls_n\overline{\phi}(y') + L\cdot\dist(x,x'), \quad\mbox{where}\quad L:=L_*\left(t+\frac{1}{1-\lambda_s^\eta}\right).
  \]   
\end{clm}

\begin{proof}
First note that $\dist(T^k y,T^ky')\le \dist(x,x')$ for $0\le k<n$. This is because $T^ky,T^ky'\in T^k\overline{J}$, $T^{n-k}: T^k\overline{J}\to \overline{I}$ is bi-Lipschitz and $|(T^{n-k})'|\ge 1$ holds Lebesgue almost everywhere. It follows that
\[
|\mcls_n\overline{\phi}(y)-\mcls_n\overline{\phi}(y')|\le L_*\cdot\sum_{0\le k<n}\dist(T^ky,T^ky')\le nL_*\cdot\dist(x,x').
\]
Then the conclusion follows when $n\le t$. Now let $n>t$. Then we have:
 \[
\begin{split}
   |\mcls_n\overline{\phi}(y)- \mcls_n\overline{\phi}(y')| &   = |\mcls_{n-t}\phi(y)-\mcls_{n-t}\phi(y')| +  |\mcls_{t}\phi(T^{n-t}y)-\mcls_{t}\phi(T^{n-t}y')| \\
     & \le |\mcls_{n-t}\phi(y)-\mcls_{n-t}\phi(y')| + tL_*\cdot\dist(x,x').
\end{split}
\]
It remains to estimate $|\mcls_{n-t}\phi(y)-\mcls_{n-t}\phi(y')|$. To this end, first note that for $0\le m<n$, either $T^m \overline{J}\subset \overline{V_s}$ or $T^m J\cap V_s =\emptyset$ holds; therefore, $T^m y\notin \overline{V_s}$ implies that $T^m J\cap V_s =\emptyset$. Then from the Claim~\ref{clm:freq} we know that 
 \[
 \#\{ 0 \le j < n-k : T^{j+k}J \cap V_s=\emptyset \} \ge \eta (n-k), \quad \forall\, 0\le k<n-t.
 \]
Combining this with \eqref{eq:exp contract} yields that 
\[
\dist(T^k y, T^k y')\le\lambda_s^{\eta (n-k)}\cdot \dist(x,x'), \quad \forall\, 0\le k<n-t.
\]
It follows that
\[
|\mcls_{n-t}\phi(y)-\mcls_{n-t}\phi(y')|\le L_*\cdot \sum_{0\le k<n-t} \dist(T^k y, T^k y') \le \frac{L_*}{1-\lambda_s^\eta}\cdot \dist(x,x'),
\]
which completes the proof.
\end{proof}

Now we are ready to complete the proof of Proposition~\ref{prop:sub}. 

Let us first show that $\psi_n\le sC_*+L$ for each $n\ge 1$. Given $x\in X$, let $y$ and $I,J$ be as given in Claim~\ref{clm:lip diff}. Since $J\in\mcli_{n+s}$, there exists $y'\in \overline{J}$ with $T^{n+s}y'=y'$. Then we have:
\[
\psi_n(x)=\mcls_n\overline{\phi}(y)\le \mcls_n\overline{\phi}(y')+ L,
\]
and 
\[
\mcls_n\overline{\phi}(y')= \mcls_{n+s}\overline{\phi}(y') - \mcls_{s}\overline{\phi}(T^n y')\le 0 + sC_*.
\]
The conclusion follows.

Finally, for each $n\ge 1$, let us show that $\lip(\psi_n)\le L$. Due to the linear structure of the unit circle, it suffices to show that for any $I\in\mcli_s$ and any $x,x'\in \overline{I}$, $\psi_n(x)\le \psi_n(x')+L\cdot\dist(x,x')$. This follows directly from Claim~\ref{clm:lip diff}.

\appendix

\section{Proof of Lemma~\ref{lem:locking}}\label{app:lock}

In this appendix we provide a self-contained proof of Lemma~\ref{lem:locking} following the mainline in \cite{BZ15}. It turns out that compactness of $X$ is redundant in the argument: as a generalization of Lemma~\ref{lem:locking}, we shall prove Proposition~\ref{prop:lock} below instead. Let us follow the terminologies and notations introduced in \S~\ref{se:intro}.
Since Lipschitz functions in a general metric space $X$ may not be bounded, let us introduce
\[
Lip^*(X):=\{f\in Lip(X) : \text{$f$ is bounded from above}\}.
\] 
Then $\int \phi\dif\nu\in [-\infty,+\infty)$ is well-defined for any $\phi\in Lip^*(X)$ and any $\nu\in \mclm_T$ (provided that $\mclm_T$ is non-empty). 

\begin{prop}\label{prop:lock}
  Let $X$ be a metric space and let $T:X\to X$ be continuous. Suppose that there exists a periodic orbit $\mclo$ of $T$ and that
  $\phi\in Lip^*(X)$ is maximized by $\mclo$. Then  for any $\ve>0$, there exists $\delta>0$ such that for each $\psi\in Lip^*(X)$ with $\lip(\psi)\le \delta$,
  $\phi-\ve\cdot\dist(\cdot,\mclo)+\psi \in Lip^*(X)$ is uniquely maximized by $\mclo$.
\end{prop}

To prove Proposition~\ref{prop:lock}, we need the facts below (corresponding to \cite[Lemma~2]{BZ15}). 

\begin{lem}\label{lem:domination}
 Let $(X,T)$ and $\mclo$ be as given in Proposition~\ref{prop:lock}. Let $p=\#\mclo$ and let $\mu=\frac{1}{p}\sum_{y\in\mclo}\delta_{y}$ denote the averaged Dirac measure supported on $\mclo$. Then there exists $C\ge 1$ such that the following hold.
\begin{itemize}
  \item [(1)] For any $x\in X$, there exists $y\in \mclo$ such that
\begin{equation}\label{eq:binding}
  \sum_{0\le k<p}\dist(T^k x,T^k y) \le C \cdot\sum_{0\le k<p}\dist(T^kx,\mclo).
\end{equation}
  \item [(2)]  For any $\psi\in Lip(X)$, we have:
\begin{equation}\label{eq:bound by dist}
  \big|\sum_{0\le k<p}\psi(T^k x) - p\int \psi\dif\mu\big|  \le C\lip(\psi) \cdot \sum_{0\le k<p} \dist(T^k x,\mclo), \quad \forall x\in X.
\end{equation}
Moreover, if $\psi\in Lip^*(X)$ additionally, then for any $\nu\in\mclm_T$, we have: 
  \[
  \big|\int\psi\dif\nu -\int\psi\dif\mu\big|\le  C\lip(\psi) \cdot \int  \dist(\cdot,\mclo)\dif\nu.
  \]
\end{itemize}

\end{lem}

\begin{proof} (1). The case $p=1$ is trivial. Now suppose $p\ge 2$ and let
\[
\Delta=\frac{1}{2}\min\{\dist(y_1,y_2):y_1,y_2\in\mclo, y_1\ne y_2\}>0.
\]
By continuity, there exists $0<\delta\le \Delta$ such that if $x\in X$ and $y\in\mclo$ satisfies $\dist(x,y)<\delta$, then $\dist(T^kx,T^ky)<\Delta$ for $0\le k<p$. Given $x\in X$, take $y\in \mclo$ such that $\dist(x,y)=\dist(x,\mclo)$.
\begin{itemize}
  \item If $\dist(x,\mclo)<\delta$, then due to the choice of $\delta$, $\dist(T^kx,T^k y)=\dist(T^kx,\mclo)$ for $0\le k<p$, and hence \eqref{eq:binding} becomes equality for $C=1$.
  \item Otherwise, $\dist(x,\mclo)\ge \delta$. On the other hand,
\[
\dist(T^k x,T^k y) \le \dist(T^k x,\mclo) + D, \quad 0\le k<p,
\]
where $D:=\max\{\dist(y_1,y_2):y_1,y_2\in\mclo\}$. Then \eqref{eq:binding} holds for $C=1+\frac{pD}{\delta}$.
\end{itemize}

(2). Note that $p\int \psi\dif\mu=\sum_{0\le k<p} \psi(T^k y),\forall y\in \mclo$. Then \eqref{eq:bound by dist} follows from  assertion (1) together with the estimate below: 
\[
|\sum_{0\le k<p}(\psi(T^k x)-\psi(T^k y))| \le \mathrm{lip}(\psi)\cdot 
\sum_{0\le k<p}\dist(T^k x,T^k y),\quad \forall x\in X,\forall y\in\mclo.
\]
The ``moreover" part follows from integrating \eqref{eq:bound by dist} with respect to $\nu$. 
\end{proof}

\begin{proof}[Proof of Proposition~\ref{prop:lock}] 
  Let $p,\mu,C$ be as given in Lemma~\ref{lem:domination}. Fixing $\ve>0$ and $\psi\in Lip^*(X)$ with $\lip(\psi)<\frac{\ve}{C}$, denote $\varphi:=\phi-\ve\cdot\dist(\cdot,\mclo)+\psi$. 
  Given $\nu\in \mclm_T$ with $\nu\ne \mu$, the goal is to show $\int \varphi\dif\nu<\int \varphi\dif\mu$. We may assume that $\int\dist(\cdot,\mclo)\dif\nu<+\infty$, because otherwise the conclusion is evident. Then 
  \[
  \int\psi \dif\nu - \int\psi \dif\mu \le C\lip(\psi)\cdot\int\dist(\cdot,\mclo)\dif\nu < \ve \int\dist(\cdot,\mclo)\dif\nu,
  \] 
where the ``$\le$" is due to the ``moreover" part in assertion (2) of Lemma~\ref{lem:domination}, and the ``$<$" is because $C\lip(\psi)<\ve$ and $\int\dist(\cdot,\mclo)\dif\nu\in (0,+\infty)$.  
Combining the displayed inequality above with $\int \phi\dif\nu \le\int \phi\dif\mu$ and  $\int\dist(\cdot,\mclo)\dif\mu=0$, the proof is done.
\end{proof}

\section{Proof of Lemma~\ref{lem:Morris}}\label{app:Morris}

This appendix is devoted to providing a self-contained proof of Lemma~\ref{lem:Morris}, which is essentially a reproduction of Morris' original proof in \cite{Mor07}. 

To begin with, note that $\gamma<+\infty$ implies that
\[
\limsup_{n \to\infty}\frac{\mcls_n\phi(x)}{n}\le 0, \quad \forall x\in X.
\]
Combining this with Birkhoff's ergodic theorem and $\int \phi\dif\mu=0$, we have:
\begin{equation}\label{eq:a.e. 0}
  \lim_{n \to\infty}\frac{\mcls_n\phi(x)}{n} = 0, \quad \mu\mbox{-a.e.}~x\in X.
\end{equation}

To complete the proof, assume that
\[
\mcls_N\phi(x_0) < \alpha
\] 
holds for some $\alpha\in\mbbr$, $x_0\in \supp \mu$ and $N\ge 1$. Then it suffices to show that $\alpha\ge - \gamma$. 

Since $T$ is continuous and $\phi$ is upper semi-continuous,  $\mcls_N\phi$ is also upper semi-continuous, so there exists an open neighborhood $V_0$ of $x_0$ such that 
\begin{equation}\label{eq:minus gamma}
  \mcls_N\phi(y) <  \alpha, \quad \forall y\in V_0.
\end{equation}
Since $x_0\in \supp \mu$, $\mu(V_0)>0$. Let 
\[
V:=\{x\in V_0: \mbox{there are infinitely many $n\ge 1$ such that  $T^n x\in V_0$} \}.
\]
Then $V$ is Borel with $\mu(V)=\mu(V_0)>0$, and $x\in V\cap T^{-n}V_0$ for some $n\ge 1$ implies that $T^n x\in V$. Let $\mu_V=\frac{1}{\mu(V)}\mu|_V$, which is a Borel probability measure on $V$. 

For $x\in V$, define its first return time $r_V$ and first return map $R_V$ to $V$ as follows:
\[
r_V(x):=\inf \{n\ge 1: T^n x \in V\}, \quad R_V(x):=T^{r_V(x)}(x).
\]
Then $r_V:V\to \mbbn$ and $R_V:V\to V$ are well-defined and Borel, and $\mu_V$ is $R_V$-invariant. Moreover, by Kac's Lemma (see, for example, \cite[Theorem~1.2.2]{VO16}),
\[
\int_V r_V\dif\mu = \mu(\cup_{n=0}^\infty T^{-n}V)\in (0,1].
\]
In particular, $r_V$ is $\mu_V$ integrable.

Given $x\in V$, define $t_0(x):=0$ and $t_{n+1}(x):=t_n(x)+r_V(T^{t_n(x)}x)$ for any $n\ge 0$ by induction on $n$. Note that $t_n(x)$ is nothing but the $n$-th Birkhoff sum of $r_V$ along the $R_V$-orbit of $x$. Then applying Birkhoff's ergodic theorem to the measure preserving system $(V,R_V,\mu_V)$ with integrable function $r_V$, we conclude that
\[
\lim_{n\to \infty} \frac{t_n(x)}{n} 
\]
exists for $\mu_V$-a.e. $x\in V$ and 
\[
\int_V \lim_{n\to \infty} \frac{t_n(x)}{n} \dif\mu_V(x) =\int_V r_V(x)\dif \mu_V(x)>0.
\]
In particular, there exists a Borel subset $W$ of $V$ with $\mu(W)>0$, such that 
\[
t(x):=\lim_{n\to \infty} \frac{t_n(x)}{n}\in (0, +\infty)
\] 
exists for any $x\in W$.
 
On the other hand, given $x\in V$, by the definition of $t_n(x)$, for any $m\ge 1$, $T^{t_{(m-1)N}(x)}(x)\in V$ and $t_{mN}(x)\ge t_{(m-1)N}(x)+N$. Combining this with \eqref{eq:minus gamma} and the definition of $\gamma$ yields that
\[
\sum_{t_{(m-1)N}(x) \le k< t_{mN}(x)} \phi(T^k x) \le \alpha+\gamma.
\]
It follows that
\[
\limsup_{m\to\infty} \frac{\mcls_{t_{mN}(x)} \phi(x)}{m} \le \alpha+\gamma, \quad \forall x\in V.
\]

Now for any $x\in W$, we have: 
\[
\begin{split}
   \liminf_{n \to\infty}\frac{\mcls_n\phi(x)}{n}  & \le \liminf_{m\to\infty} \frac{\mcls_{t_{mN}(x)} \phi(x)}{t_{mN}(x)} \\
     & = \frac{\liminf\limits_{m\to\infty}\frac{\mcls_{t_{mN}(x)}\phi(x)}{m}} {N\cdot \lim\limits_{m\to\infty}\frac{t_{mN}(x)}{mN}} \\
     & \le \frac{\alpha+\gamma}{N\cdot t(x)}.
\end{split}
\]
Combining this with $t(x)<+\infty$ and \eqref{eq:a.e. 0}, we conclude that $\alpha+\gamma\ge 0$, which completes the proof.

\medskip
\noindent {\bf Acknowledgements.}
We thank the anonymous referees for valuable suggestions. BG is partially supported by the Scientific Research Fund of the Education Department of Hunan Province (No. 22B0024), and by National Natural Science Foundation of China (No. 12071118). RG is partially supported by the Taishan Scholar Project of Shandong Province (No. tsqnz20230614).

\bibliographystyle{plain}             
\bibliography{refer}

\end{document}